\documentclass[1p]{preprint}

\usepackage{hyperref}
\usepackage{amsmath,amssymb,amsthm}
\usepackage{pdfsync}
\usepackage{subcaption}

\usepackage{graphicx}
\graphicspath{{./figures/}}
\usepackage[skins]{tcolorbox}
\usepackage[]{algorithm2e}
\usepackage{color}
\usepackage{pdfsync}
\usepackage{mathtools}
\usepackage[mathscr]{eucal}
\usepackage{caption}
\usepackage{anslistings}
\usepackage{booktabs}
\usepackage{colortbl}
\usepackage{pgfplots}
\usepackage{pgfplotstable}
\usepackage{tikz}
\usepackage{tkz-euclide}
\usetikzlibrary{calc}
\usepackage{empheq}
\usepackage{wrapfig}
\usetikzlibrary{decorations.pathreplacing,shapes.misc,calc,intersections,arrows,external}
\usepgfplotslibrary{external}
\usepackage{xifthen}

\usepackage[normalem]{ulem}

\usepackage[nottoc]{tocbibind}

\usepackage[refpage,noprefix,intoc]{nomencl}

\makenomenclature

\journal{}

\usepackage[draft,inline,marginclue,mode=multiuser]{fixme}

\FXRegisterAuthor{lh}{alh}{LH}
\FXRegisterAuthor{ga}{aga}{GA}

\usepackage{definitions-private}

\begin{document}

\begin{frontmatter}

  \title{Multiscale modeling of fiber reinforced materials via non-matching immersed methods}
  \author[sissa]{Giovanni Alzetta}
  \ead{giovannialzetta@gmail.com}
  \author[sissa]{Luca Heltai}
  \ead{luca.heltai@sissa.it}

  \address[sissa]{SISSA-International School for Advanced Studies\\
    via Bonomea 265, 34136 Trieste - Italy}


  \begin{abstract}
    Fiber reinforced materials (FRMs) can be modeled as bi-phasic
    materials, where different constitutive behaviors are associated
    with different phases. The numerical study of FRMs through a full
    geometrical resolution of the two phases is often computationally
    infeasible, and therefore most works on the subject resort to
    homogenization theory, and exploit strong regularity assumptions on
    the fibers distribution. Both approaches fall short in intermediate
    regimes where lack of regularity does not justify a homogenized
    approach, and when the fiber geometry or their numerosity render
    the fully resolved problem numerically intractable.

    In this paper, we propose a distributed Lagrange multiplier
    approach, where the effect of the fibers is superimposed on a
    background isotropic material through an independent description of
    the fibers.  The two phases are coupled through a constraint
    condition, opening the way for intricate fiber-bulk couplings as
    well as allowing complex geometries with no alignment requirements
    between the discretisation of the background elastic matrix and the
    fibers.

    We analyze both a full order coupling, where the elastic matrix is
    coupled with fibers that have a finite thickness, as well as a
    reduced order model, where the position of their centerline uniquely
    determines the fibers. Well posedness, existence, and uniquess of
    solutions are shown both for the continuous models, and for the
    finite element discretizations. We validate our approach against the
    models derived by the rule of mixtures, and by the Halpin-Tsai
    formulation.
  \end{abstract}


\end{frontmatter}


\section{Introduction}
\label{sec:introduction}

Numerous engineering applications require the efficient solution of
partial differential equations involving multiple, complex geometries
on different phases; \emph{composite materials} are the
prototypical example of such problems. During the past fifty years,
the interest in composite materials flourished multiple times;
it began for their applications to new materials
in multiple fields, such as aerospace engineering
\cite{hashin1972theory}, civil engineering \cite{mehta1986concrete},
and materials science \cite{behzad2007finite, mallick2007fiber}.

During the nineties, the increasing importance of biomechanics in
life sciences lead to the development of numerous models describing, e.g.,
arterial walls \cite{holzapfel2000new}, soft tissues
\cite{holzapfel2001biomechanics}, and muscle fibers
\cite{huijing1999muscle}.

Recent years saw the rise of new application fields, such as
the study of natural fiber composites \cite{pickering2016review},
and engineering methods to accurately
recover the three-dimensional structure of a material sample, e.g.,
\cite{iizuka2019development,konopczynski2019fully}.

From the first studies on composites, it has been clear that their
properties are strongly dependent on their internal structure: the
volume ratio between each component, the orientation, the shape, all
contribute substantially to the material's properties
\cite{hashin1972theory,hashin1983analysis}. One of the most
wide-spread and significant example of composites is that of Fiber
Reinforced Materials (FRMs), where thin, elongated structures (the
fibers) are immersed in an underlying isotropic material (the elastic
matrix).

We may separate the approaches used to study FRMs into two broad
groups: i) ``\emph{homogenization methods}'', which study a complex
inhomogeneous body by approximating it with a fictitious homogeneous
body that behaves globally in the same way~\cite{zaoui2002continuum},
and ii) ``\emph{fully resolved}'' methods, which use separate
geometrical and constitutive descriptions for the elastic matrix and
the individual fibers.

As examples of analytical ``\emph{homogenization methods}'' we recall
the rule of mixtures~\citep{hashin1965elastic} and the
\emph{empirical} Halpin-Tsai equations~\cite{halpin1976halpin}, used
to study a transversely isotropic unidirectional composite, where
fibers are uniformly distributed and share the same orientation. The
development of homogenization theory led, in recent years, to more
complex models, e.g., \cite{holzapfel2001viscoelastic,
  barrage2019modelling}.

More intricate homogeneization approaches rely on numerical methods to
provide a ``cell'' behaviour, which is then replicated using
periodicity, using, e.g., the Finite Element Method
\cite{adams1967transverse,nakamura1993effects,lu20143d}, Fourier
transforms
\cite{moulinec1998numerical,moulinec2003comparison,eloh2019development},
or Stochastic Methods \cite{lucas2015stochastic}.

The fundamental limit of all ``\emph{homogenization methods}'' approaches
is the impossibility of adapting them to study composites with little regularity.
In these cases, the different phases are typically modeled separately, as a continuum.
This approach began with Pipkin \cite{Pipkin1968} on two dimensional membranes, and
was then expanded to three dimensional examples by others (see, e.g.,
\cite{Kyriacou1996} for a detailed bibliography).

Fully resolved methods allow richer structures, but require
a high numerical resolution, especially when
material phases have different scales. The complex meshing and coupling
often result in an unbearable computational cost,
limiting the use of these methods.

The purpose of this paper is to introduce an approach which is fit for
materials that have \emph{intermediate} properties, i.e., they possess no
particular regularity, and are made by a relatively high number of fiber
components. Similarly to~\cite{radtke2010computational}, our model is inspired
by the Immersed Boundary Method (IBM)~\cite{peskin2002}, and by its variational
counterparts~\cite{boffi2008hyper, Heltai-2008-a, HeltaiCostanzo-2012-a,
  RoyHeltaiCostanzo-2015-a, HeltaiRotundo-2018-a}, where the elastic matrix and
the fibers are modeled independently, and coupled through a non-slip condition.
The model we present can be interepreted as a variation of the embedded
reinforced method (see~\cite{goudarzi2019discrete} and the references therein),
where we aim at providing an efficient numerical method for FRMs that allows
the modeling of complex networks of fibers, where one may also be interested in
the elastic properties of single fibers, without requiring the resolution of
the single fibers in the background elastic matrix. From the computational
point of view, this approach allows the use of two independent discretizations:
one describing the fibers, and one describing the \emph{whole domain}, i.e.,
both the elastic matrix and the fibers. A distributed Lagrange multiplier is
used to couple the independent grids, following the same spirit of the finite
element immersed boundary method~\cite{boffi2003finite, boffi2008hyper},
separating the Cauchy stress of the whole material into a background uniform
behavior and into an \emph{excess} elastic behavior on the fibers.

Section \ref{sec:sec1} introduces the classical fully resolved model
of a collection of fibers immersed in an elastic matrix. For
simplicity, we do not include dissipative terms, and restrict our
study to linearly elastic materials.  The problem is then reformulated
exploiting classical results of mixed methods (see Chapter 4 of
\cite{boffi2013mixed}), following ideas similar to those found in
\cite{boffi2017fictitious}, proving that both the continuous and
discrete formulations we propose are well-posed with a unique
solution.

The use of a full three dimensional model for the fibers still results in high
computational costs; the obvious simplification would be to approximate the
fibers with one-dimensional structures. This approach is non-trivial because
the theoretical solution of the fully resolved variational problem does not
posses enough Sobolev regularity to allow the restriction of a
three-dimensional field to a one-dimensional domain. A possible workaround
involves the use of weighted Sobolev spaces, combinded with graded
meshes~\cite{d2008coupling, d2012finite}. This approach remains unfeasible when
the number of fibers is large. In Section \ref{sec:sec3}, we propose and
analyze an alternative strategy, where under some additional assumptions we
construct a $3D-1D$ coupling that relies on local averaging techniques. A
similar procedure is used in~\cite{heltaimultiscale} to model vascularized
tissues. To conclude, we validate our thin fiber model in Section
\ref{sec:num-exp}, and draw some conclusions in Section~\ref{sec:conclusions}.

\section{Three-dimensional model}
\label{sec:sec1}

\begin{figure}
  \centering
  \includegraphics[width=.5\textwidth]{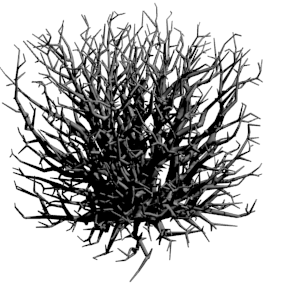}
  \caption{An example of a fiber structure for which the mesh
    generation for the fibers would be trivial, but the resulting
    three-dimensional elastic matrix would be much more expensive to
    resolve in full.}
  \label{fig:fiber-tree}
\end{figure}

Many bi-phasic materials present a relatively simple fiber structure
but result in a very intricate elastic matrix.  Consider, for example,
Figure~\ref{fig:fiber-tree}: constructing a discretization grid for
the fibers themselves maybe simple enough, but building a fully
resolved grid for the \emph{surrounding} elastic matrix, in this case,
may require eccessive resolution, and result in a computationally hard
problem to solve. We wish to describe a new approach, where we
substitute the complex mesh needed for the elastic matrix with a
simple one describing the whole domain, and overlap the fiber
structure independently with respect to the background grid, and
couple the two systems via distributed Lagrange multipliers.

\subsection{Problem formulation}
\label{sec:problem-formulation-3d-3d}

As a model bi-phasic material, we consider a linearly elastic fiber reinforced
material, in the quasi-static small strain regime. We consider continuous
fibers with a perfect bond between the two phases (see~\cite[Chapter
16]{callister2007materials}), leading to a no-slip condition between the fibers
and the elastic matrix. The extension to finite strain elasticity, dynamic
problems, or non-perfect bonds do not present additional difficulties and will
be the subject of future investigations.

To describe the composite, we use a connected, bounded, Lipschitz
domain $\Omega \subset \mathbb{R}^d$ of dimension $d=3$, composed of a
fiber phase $\Omega_f \subset \Omega$, and an elastic matrix
$\hat{\Omega} \coloneqq \Omega \setminus \Omega_f \subset
  \mathbb{R}^d$, which we assume to be a connected, Lipschitz domain. We
describe each of the $n_f \in \mathbb{N}$ fibers with a connected,
Lipschitz domain and $\Omega_f$ is obtained as the union of these
(possibly overlapping) domains.

\begin{figure}[htb]
  \begin{subfigure}{0.32\textwidth}
    \begin{tikzpicture}[scale = 1.2]
      \draw (0, 0) rectangle (3, 2);
      \node [below] at (1.5,0) {$\hat{\Omega}$};
      \foreach \i in {1,...,9}
        {
          \draw[double=none, double distance=3pt, line join=round, line cap=round, draw=black] (0.2,0.2*\i) -- (2.8,0.2*\i);
        }
    \end{tikzpicture}
    \caption*{The elastic matrix}
  \end{subfigure}\hfill
  \begin{subfigure}{0.32\textwidth}%
    \begin{tikzpicture}[scale = 1.2]
      \draw[dotted,opacity=0.1] (0, 0) rectangle (3, 2);
      \node [below] at (1.5,0) {$\Omega_f$};
      \foreach \i in {1,...,9}
        {
          \draw[double distance=0pt, line join=round, line cap=round, draw=black, ultra thick] (0.2,0.2*\i) -- (2.8,0.2*\i);
        }
    \end{tikzpicture}
    \caption*{The fibers}
  \end{subfigure}\hfill
  \begin{subfigure}{0.32\textwidth}%
    \begin{tikzpicture}[scale = 1.2]
      \draw (0, 0) rectangle (3, 2);
      \node [below] at (1.5,0) {$\Omega$};
      \foreach \i in {1,...,9}
        {
          \draw[double distance=0pt, line join=round, line cap=round, draw=black, ultra thick] (0.2,0.2*\i) -- (2.8,0.2*\i);
        }
    \end{tikzpicture}
    \caption*{The material}
  \end{subfigure}
  \caption*{Example of a two-dimensional section of an FRM with uniformly oriented fibers.}
  \label{fig:composite2d}
\end{figure}

\begin{remark}
  The results of this section hold for a general domain
  $\Omega_f$, union of multiple components with the required
  regularity. The property of the
  fibers of being thin, elongated, structures, is only needed for the
  model of Section \ref{sec:sec3}, and plays no role in this section.
\end{remark}

Given a continuous displacement field $u \colon \Omega \rightarrow
\mathbb{R}^d$, representing a deformation from the equilibrium configuration,
the corresponding stress tensor on $\Omega$ can be expressed using the
stress-strain law \cite{Gurtin1981}:
\begin{equation*}
  S[u] = \mathbb{C}\nabla u,
\end{equation*}
where $\mathbb{C}$ is a symmetric $4^{th}$ order tensor that takes the form:
\begin{equation*}
  \mathbb{C} =
  \begin{cases}
    \mathbb{C}_{\Omega} & \qquad \text{in } \hat{\Omega}, \\
    \mathbb{C}_f        & \qquad \text{in } \Omega_f.
  \end{cases}
  \numberthis \label{eqC}
\end{equation*}

Here $\mathbb{C}_{\Omega}$ and $\mathbb{C}_f$ are assumed to be constant over
their respective domains, and represent the elasticity tensors of the elastic
matrix and of the fibers. The assumption of perfect bonds between the fibers
and the elastic matrix (see~\cite[Chapter 16]{callister2007materials}) allows
one to define the full order problem as a single elasticity problem on the
union of the two domains, with pecewise elastic properties. This assumption may
be relaxed by reformulating the fiber elasticity equations and the elastic
matrix elasticity equations separately, and coupling them with appropriate
interface (or bonding) conditions.

The classical formulation of static linear elasticity can be thought of as a
force balance equation (see, for example,~\cite{Gurtin1981}):

\begin{problem}[Classic Strong Formulation]
\label{pb:0}
Given an external force density field $b$, find the displacement $u$
such that
\begin{equation}
  \begin{aligned}
    - \Div (\mathbb{C} u) & = b       &  & \text{ in } \Omega,          \\
    u                     & = 0 \quad &  & \text{ on } \partial \Omega.
  \end{aligned}\label{eq:strong-formulation}
\end{equation}
\end{problem}

Due to the piecewise nature of $\mathbb{C}$, it is natural to
reformulate Problem~\ref{pb:0} into a variational or weak
formulation. Given a subset $D$ of $\partial \Omega$, we introduce the
notation for the subspace of the Sobolev space $H^1(\Omega)$
with functions vanishing
on a subset $D$ of the boundary $\partial \Omega$:
$$
  H^1_{0,D}(\Omega)^d \coloneqq
  \lbrace v \in H^1 (\Omega)^d
  \colon \restr{v}{D} = 0\rbrace,
$$
with norm
$\|\cdot\|_V = \|\cdot\|_{H^1(\Omega)} \coloneqq \|\cdot\|_\Omega +
  \|\nabla \cdot\|_\Omega$, where the symbol $\|\cdot\|_A$ represents
the $L^2(A)$ norm over the measurable set $A \subset \Omega$, and
$(\cdot,\cdot)_A$ represents the $L^2$ scalar product on the given
domain $A$.
We define the following space:
\begin{align*}
  V & \coloneqq \left( H_{0,\partial \Omega}^1 (\Omega) \right)^d
  = \lbrace v \in H^1 (\Omega)^d
  \colon \restr{v}{\partial \Omega} = 0\rbrace,
\end{align*}
The standard weak formulation reads:

\begin{problem}[Classic Weak Formulation]
\label{pb:1}
Given $b\in L^2(\Omega)^d$,
find $u \in V$ such that:
\begin{equation}
  \label{eq:classicweak}
  (\mathbb{C}\nabla u, \nabla v)_\Omega
  =
  (b,v)_\Omega \qquad \forall v\ \in V.
\end{equation}
\end{problem}

The main idea behind our reformulation is to rewrite
Problem~\ref{pb:1} into an equivalent form, where we define two
independent functional spaces. The novelty we introduce is to define
the functional spaces on $\Omega$ and $\Omega_f$, and not on
$\hat{\Omega}$ and $\Omega_f$. To achieved this, we define two
fictitious materials: one with the same properties of the elastic
matrix, occupying the full space $\Omega$, and one describing the
``excess elasticity'' of the fibers separately, defined on $\Omega_f$
only.  The first step in this direction is to split the left-hand side
of Equation~\eqref{eq:classicweak} on the two domains:

\begin{align*}
  (\mathbb{C}\nabla u, \nabla v)_\Omega
   & =
  (\mathbb{C}_f \nabla u, \nabla v)_{\Omega_f} +
  (\mathbb{C}_\Omega \nabla u, \nabla v)_{\hat{\Omega}} \\
   & = (\mathbb{C}_f \nabla u, \nabla v)_{\Omega_f}
  + \underbrace{(\mathbb{C}_\Omega \nabla u, \nabla v)_{\hat{\Omega}}
    + (\mathbb{C}_\Omega\nabla u, \nabla v)_{\Omega_f} }_{\Omega}
  - (\mathbb{C}_\Omega\nabla u, \nabla v)_{\Omega_f}    \\
   & =
  (\mathbb{C}_\Omega \nabla u, \nabla v)_{\Omega} +
  ( \delta \mathbb{C}_f \nabla u, \nabla v)_{\Omega_f},
\end{align*}
where $\delta \mathbb{C}_f \coloneqq \mathbb{C}_f - \restr{(\mathbb{C}_\Omega)}{\Omega_f}$.

For simplicity, we improperly use the expression
``elastic matrix equation'' and ``fiber equation'', even though they
should be really considered as the ``whole domain equation'', and a
``delta fiber equation''.
This formal separation does not change the original variational
problem, which can still be stated explicitly: given
$b\in L^2(\Omega)^d$, find $u \in V$ such that:
\begin{equation}
  \label{eq:ofpb2}
  (\mathbb{C}_\Omega \nabla u, \nabla v)_{\Omega} +
  ( \delta \mathbb{C}_f \nabla u, \nabla v)_{\Omega_f}
  =
  (b,v)_\Omega \qquad \forall v\ \in V.
\end{equation}

The boundary of the domain $\Omega$ induces a natural splitting on the
boundary of the fibers: we define the following partition of
$\partial \Omega_f$:
\begin{align}
  B_i & \coloneqq \partial \Omega_f \setminus \partial \Omega \\
  B_e & \coloneqq \partial \Omega \cap \partial \Omega_f,
  \label{def:B_e}
\end{align}
where $B_i$ is the \textbf{i}nterface between the fibers and the
elastic matrix, while $B_e$ is the interface between the fibers the
\textbf{e}xterior part of $\Omega$, that lies on the boundary
$\partial \Omega$.
Next we define the restriction of $H_0^1(\Omega)^d$ on the fibers:
\begin{align}
  W & \coloneqq \left( H^1_{0,B_e}(\Omega_f) \right)^d,
\end{align}
and separate the solution of Problem \ref{pb:1} in two components,
one describing the whole material, the other describing only the fibers.

When reformulating the problem with two independent variables, the perfect bond
condition between the fibers and the elastic matrix must be imposed via a
volumetric non-slip constraint on the solution $(u,w) \in V \times W$:

\begin{equation} \label{cond:non-slip}
  \restr{u}{\Omega_f} = w.
\end{equation}

The described setting is similar to the distributed Lagrange multiplier method
used to model fluid structure interaction problems with non-matching
discretisations, as described in~\cite{boffi2003finite,
  auricchio2015fictitious, boffi2014mixed}.

The result is a constrained minimization problem:
\begin{equation}
  \label{eq:argmax}
  u,w = \arg \inf_{\substack{(u, v) \in V\times W \\
      u |_{\Omega_f}= w}} \psi(u,w),
\end{equation}
where we defined the total elastic energy of the system as
\begin{equation}
  \label{eq:energy(u,w)}
  \psi(u,w) = \frac{1}{2}(\mathbb{C}_\Omega \nabla u,\nabla u)_\Omega
  + \frac{1}{2}(\delta \mathbb{C}_f \nabla w,\nabla w)_{\Omega_f}
  - (b,u)_\Omega.
\end{equation}

To impose the non-slip constraint of Equation~\eqref{cond:non-slip}, we use the
duality product $W'\times W$ as in \cite{boffi2017fictitious}. We define $Q
  \coloneqq W'$, and enforce the perfect bond on the fibers by asking that
\begin{equation}
  \label{eq:weak-non-slip}
  \langle q, \restr{u}{\Omega_f}-w \rangle_{Q \times W} = 0 \qquad \forall q \in Q.
\end{equation}
For simplicity we will omit the subscript $Q \times W$ from now on.

The constrained minimization expressed in Equation~\eqref{eq:argmax} is
equivalent to the saddle point problem:
\begin{equation}
  \label{eq:argmaxargmin}
  u,w,\lambda = \arg \inf_{\substack{u\in V \\ w\in W}}
  \left( \arg \sup_{\substack{\lambda \in Q}}
  \psi(u,w,\lambda) \right),
\end{equation}
where the constraint is imposed weakly as in~\ref{eq:weak-non-slip}
through a Lagrange multiplier:

\begin{equation*}
  \psi(u,w,\lambda) := \frac{1}{2}(\mathbb{C}_\Omega \nabla u,\nabla u)_\Omega
  + \frac{1}{2}(\delta \mathbb{C}_f \nabla w,\nabla w)_{\Omega_f}
  + \langle \lambda, \restr{u}{\Omega_f}-w\rangle
  - (b,u)_\Omega.
\end{equation*}

A solution to Equation~\eqref{eq:argmaxargmin} is obtained by solving
the Euler-Lagrange equation:
\begin{equation}
  \label{eq:variational1}
  \left\langle D_u \psi,v \right\rangle
  + \left\langle D_w \psi,y \right\rangle
  + \left\langle D_{\lambda} \psi,q \right\rangle
  = 0 \qquad \forall\ v \in V,\ \forall\ y \in W,\forall\ q \in Q,
\end{equation}
that is:
\begin{problem}[Saddle Point Weak Formulation]
\label{pb:4}
Given $b \in L^2(\Omega)^d$,
find $u \in V, w \in W, \lambda \in Q$ such that:
\begin{align}
  \label{eq:WeakForm2}
  (\mathbb{C}_\Omega\nabla u, \nabla v)_\Omega
  +  \langle \lambda,  \restr{v}{\Omega_f} \rangle & =
  (b,v)_\Omega \qquad                              & \forall v \in V                    \\
  \label{eq:WeakForm2.2}
  (\delta \mathbb{C}_f \nabla w, \nabla y)_{\Omega_f}
  - \langle \lambda, y \rangle                     & = 0 \qquad      & \forall y \in W  \\
  \langle q,  \restr{u}{\Omega_f} - w  \rangle     & = 0 \qquad      & \forall q \in Q,
\end{align}
or, equivalently,
\begin{equation}
  \begin{aligned}
     &  & \mathcal{K}_{\Omega}u &  & \quad           &  & +\mathcal{B}^T \lambda &  & = (b,\cdot)_\Omega \qquad &  & \text{ in } V'  \\
     &  & \qquad                &  & \mathcal{K}_f w &  & -\mathcal{M}^T \lambda &  & =  0 \qquad \qquad        &  & \text{ in } W'  \\
     &  & \mathcal{B}u          &  & -\mathcal{M}w   &  & \qquad                 &  & = 0 \qquad \qquad         &  & \text{ in } Q',
  \end{aligned}
\end{equation}
where
\begin{equation}
  \begin{aligned}
    \mathcal{K}_{\Omega} & \colon V \rightarrow V' \qquad
                         & \langle \mathcal{K}_{\Omega} u,\cdot\rangle
                         & \coloneqq (\mathbb{C}_{\Omega}\nabla u, \nabla \cdot)_\Omega                 \\
    \mathcal{K}_f        & \colon W \rightarrow W' \qquad
                         & \langle \mathcal{K}_f w,\cdot\rangle
                         & \coloneqq (\delta \mathbb{C}_f\nabla w, \nabla \cdot)_{\Omega_f}             \\
    \mathcal{B}          & \colon V \rightarrow Q' \qquad
                         & \langle \mathcal{B}u,	\cdot\rangle                                & \coloneqq
    \langle \cdot, \restr{u}{\Omega_f} \rangle                                                          \\
    \mathcal{M}          & \colon W \rightarrow Q' \qquad
                         & \langle \mathcal{M}w \cdot\rangle                                & \coloneqq
    \langle \cdot, w \rangle.
  \end{aligned}
  \numberthis \label{operatorList}
\end{equation}
\end{problem}

\subsection{Well-posedness, existence and uniqueness}
\label{subsection:well-posed_3d3d}

Existence and uniqueness follow from standard saddle point
theory~\cite{boffi2013mixed}. We introduce the product Hilbert space, with
its norm:
\begin{align*}
  \mathbb{V}               & \coloneqq V \times W,            \\
  \|(u,w)\|^2_{\mathbb{V}} & \coloneqq \|u\|_V^2 + \|w\|_W^2,
\end{align*}
and we indicate with $\pmb{u} \coloneqq (u,w), \pmb{v} \coloneqq (v,y)$
the elements of $\mathbb{V}$. We define the bilinear forms
\begin{align*}
   & \begin{aligned}
    \mathbb{F} \colon & \mathbb{V} \times \mathbb{V} \longrightarrow \mathbb{R} \\
                      & (\pmb{u},\pmb{v})  \longmapsto
    \langle \mathcal{K}_{\Omega}u,v\rangle + \langle \mathcal{K}_f w,y\rangle = (\mathbb{C}_{\Omega}\nabla u, \nabla v)_\Omega +
    (\delta \mathbb{C}_f \nabla w, \nabla y)_{\Omega_f},
  \end{aligned} \\
   & \begin{aligned}
    \mathbb{E} \colon & \mathbb{V} \times Q \longrightarrow \mathbb{R}                                         \\
                      & (\pmb{u},q)  \longmapsto \langle \mathcal{B}u,q\rangle - \langle \mathcal{M}w,q\rangle
    = \langle q,  \restr{v}{\Omega_f}-w \rangle,
  \end{aligned}
\end{align*}
with which we reformulate the problem as: find $\pmb{u} \in \mathbb{V}, \lambda \in Q$ such that
\begin{equation}
  \begin{aligned}
    \mathbb{F}(\pmb{u},\pmb{v}) + \mathbb{E}(\pmb{v},\lambda) & = (b,v)_\Omega
    \qquad                                                    & \forall\ \pmb{v} \coloneqq (v,y) \in \mathbb{V},                     \\
    \mathbb{E}(\pmb{u},q)                                     & = 0                                              & \forall\ q \in Q.
  \end{aligned}
  \label{weakFE}
\end{equation}

\begin{proposition}
  \label{prop:inf-supE}
  There exists a constant $\alpha_1 > 0$ such that:
  $$
    \inf_{q \in Q} \sup_{(v,w) \in \mathbb{V}}
    \frac{ \langle q, \restr{v}{\Omega_f}-w \rangle}
    { \|(v,w)\|_{\mathbb{V}} ~\|q\|_{Q}}
    \geq \alpha_1.
  $$
  Moreover $\alpha_1 = 1$.
\end{proposition}

The proof for this proposition, and its discrete version, are variations
on the one found in \cite{boffi2017fictitious}.

\begin{proof}

  The non slip condition is given by the duality pairing between $Q = W'$ and $W$;
  by definition of the norm in the dual space $Q$:
  \begin{align*}
    \|q\|_{Q}
     & = \sup_{ w \in W }  \frac{ \langle q , w \rangle}{\| w\|_{Q}} \\
     & \leq \sup_{v \in V, w \in W }
    \frac{\langle q ,  \restr{v}{\Omega_f} - w \rangle}
    {(\|w\|_{Q}^2+\|v\|_V^2)^{\frac{1}{2}}},
  \end{align*}
  where the last inequality can be proven fixing $v=0$.
  The final
  statement is found dividing by $\|q\|_{Q}$ and taking the $\inf_{q \in Q}$.
\end{proof}
Notice that $\alpha_1 = 1$, and does not depend on the two domains or on the two
spaces.

\begin{proposition}
  \label{prop:inf-supF}
  Assume $\mathbb{C}_{\Omega}$ and $\mathbb{C}_f$ to be strongly elliptic with
  constants $c_\Omega$ and $c_f$ respectively, such that $c_f > c_\Omega >0$;
  then there exists a constant $\alpha_2 > 0$ such that:
  $$
    \inf_{(u,w) \in \ker{\mathbb{E}}}
    \sup_{(v,y) \in \ker{\mathbb{E}}}
    \frac{(\mathbb{C}_\Omega \nabla u, \nabla v)_{\Omega}
    + (\delta \mathbb{C}_f\nabla w, \nabla y)_{\Omega_f}}
    { \|(v,y)\|_{\mathbb{V}} \|(u,w)\|_{\mathbb{V}}}
    \geq \alpha_2.
  $$
\end{proposition}

\begin{proof}
  An immediate consequence of the hypotheses is
  that $\delta \mathbb{C}_f$ is elliptic of constant $c_f - c_\Omega$.
  Following the proof for a similar statement found in
  \cite{auricchio2015fictitious, boffi2014mixed}, given an element $(v,y)\in \ker(\mathbb{E})$,
  the fact that $\restr{v}{\Omega_a} = y$ allows to use the Poincar\'e inequality
  on $v$ to control the norm of $y$;
  for every $(u,w) \in \ker(\mathbb{E})$:
  \begin{align*}
    \sup_{(v,y) \in \ker(\mathbb{E})}
     & \frac{
    (\mathbb{C}_\Omega \nabla u, \nabla v)_\Omega
    + (\delta \mathbb{C}_f\nabla w, \nabla y)_{\Omega_f}
    }
    {\|(v,y)\|_{\mathbb{V}}}                 \\
     & \geq
    \frac{
    c_\Omega  (\nabla u,\nabla u)_{\Omega}
    + (c_f - c_\Omega ) (\nabla w,\nabla w)_{\Omega_f}
    }
    {\|(u,w)\|_{\mathbb{V}}}                 \\
     & \geq
    \frac{c_p
      \min(c_\Omega,c_f - c_\Omega)
    }{2}
    \frac{
    (u, u)_{H^1(\Omega)}
    + (w, w)_{H^1(\Omega_f)}
    }
    {\|(u,w)\|_{\mathbb{V}}}                 \\
     & \geq \alpha_2 \|(u,w)\|_{\mathbb{V}},
  \end{align*}
where we used the Poincar\'e inequality with its positive constant $c_p$ on $u \in V$,
  with $\alpha_2~\coloneqq~\frac{c_p
      \min(c_\Omega,c_f - c_\Omega)
    }{2}$, and we used the scalar product of $H^1(\Omega)^d$:
  $$
    (u, u)_{H^1(\Omega)} \coloneqq (u, v)_{\Omega} +(\nabla u, \nabla v)_{\Omega},
  $$
  and the analogous one for $H^1(\Omega_f)^d$.
  The final statement
  is obtained dividing by $\|(u,w)\|_{\mathbb{V}}$ and considering the
  $\inf_{(u,w) \in \ker{\mathbb{E}}}$.
\end{proof}

\begin{remark}
  This paper does not intend to focus on the \emph{choice}
  of elastic tensors.
  Strong ellipticity is a common property
  among them, and holds
  in the case of linearly elastic materials
  (see e.g. \cite{marsden1994mathematical}): let $u \in V$, $w \in W$
  \begin{align}
    \label{eq:elasticity_tens}
    \mathbb{C}_\Omega \nabla u   & := 2\mu_\Omega Eu + \lambda_\Omega (tr \nabla u) I = 2\mu_\Omega Eu + \lambda_\Omega (\Div u) I \numberthis \\
    \mathbb{C}_f \nabla w        & := 2\mu_f Ew + \lambda_f (tr \nabla w) I = 2\mu_f Ew + \lambda_f (\Div w)I                                  \\
    \label{eq:elasticity_tensf}
    \delta \mathbb{C}_f \nabla w & := 2(\mu_f - \mu_\Omega) Ew + (\lambda_f - \lambda_\Omega) (tr \nabla w) I                                  \\
                                 & = 2\mu_\delta Ew + \lambda_\delta (\Div
    w) I, \nonumber
  \end{align}
  where $\mu_\delta \coloneqq \mu_f - \mu_\Omega$ and $\lambda_\delta \coloneqq \lambda_f - \lambda_\Omega$,
   and $Eu = \frac{\nabla u + \nabla u^T}{2}$ is the symmetric gradient.
  These are the elastic tensors we use in Section~\ref{sec:num-exp},
  for our numerical tests.
\end{remark}

Propositions \ref{prop:inf-supF} and \ref{prop:inf-supE} imply that the inf-sup
conditions are satisfied, and that Problem~\ref{pb:4} is well-posed, and has a
unique solution.

\subsection{Finite element discretization}
\label{subsec:FEDiscretization}

The distributed Lagrange formulation of Problem~\ref{pb:4} makes it possible to
use independent triangulations for its numerical solution. Consider the family
$\mathcal{T}_h(\Omega)$ of regular meshes in $\Omega$, and a family
$\mathcal{T}_h(\Omega_f)$ of regular meshes in $\Omega_f$, where we denote by
$h$ the maximum diameter of the elements of the two triangulations. We assume
that no geometrical error is committed when meshing, i.e., $\Omega =
  \bigcup_{T_h \in \mathcal{T}_h(\Omega)} T_h $, and $\Omega_f = \bigcup_{S_h \in
    \mathcal{T}_h(\Omega_f)} S_h $. We consider two independent finite element
spaces $V_h\subset V$, and $W_h\subset W$, and we take $Q_h = W_h$. All finite
dimensional spaces are endowed with the norms of their continuous counterparts.

\begin{problem}[Discrete Weak Formulation]
\label{problem:weak3d3d_h}
Given $b \in L^2(\Omega)^d$, find $u_h \in V_h,~w_h \in W_h, ~\lambda_h \in Q_h$ such that:
\begin{align}
  (\mathbb{C}_\Omega\nabla u_h, \nabla v_h)_\Omega
  + \langle \lambda_h, \restr{v_h}{\Omega_f}\rangle & =
  (b,v_h)_\Omega \qquad               & \forall v_h \in V_h,                        \\
  (\delta \mathbb{C}_{f} \nabla w_h, \nabla y_h)_{\Omega_f}
  - \langle\lambda_h, y_h\rangle      & = 0 \qquad           & \forall y_h \in W_h, \\
  \langle q_h, \restr{u_h}{\Omega_f} - w_h\rangle   & = 0 \qquad           & \forall q_h \in Q_h.
\end{align}
\end{problem}
Existence, well posedness, and convergence are guaranteed if the
classical discrete inf-sup conditions are satisfied:

\begin{proposition}
  \label{prop:inf-supEh}
Assume that the $L^2$ projection
$$
  P_W \colon W \rightarrow W_h \subset L^2(\Omega_f)
$$
is continuous and $H^1$-stable, i.e., there exists a positive constant $c_W$ such
that for all $w \in W$:
\begin{equation}
  \label{eq:h1-stable}
  \|\nabla (P_W w) \|_{\Omega_f} \leq c_W \|\nabla w\|_{\Omega_f}.
\end{equation}
  Then there exists a constant $\alpha_{3} > 0$, independent of $h$, such that:
  $$
    \inf_{q_h \in Q_h} \sup_{(v_h,w_h) \in \mathbb{V}_h}
    \frac{\langle q_h, \restr{v_h}{\Omega_f}-w_h\rangle}
    { \|(v_h,w_h)\|_{\mathbb{V}} \|q_h\|_{Q}}
    \geq \alpha_{3}.
  $$
\end{proposition}

\begin{proof}
  For every $q_h \in Q_h$,
  by definition of the $Q$ norm, and by property~\eqref{eq:h1-stable},
  there exists $\hat{w} \in W$ such that:
  \begin{equation}
    \label{eq:proof_infsup}
    \| q_h \|_Q
    = \sup_{ w \in W }  \frac{\langle q_h , w \rangle}{\| w\|_{W}}
    = \frac{\langle q_h , \hat{w} \rangle}{\| \hat{w}\|_{W}}
    = \frac{\langle q_h , P_W \hat{w} \rangle}{\| \hat{w}\|_{W}}
    \leq c_W
    \frac{\langle q_h , P_W \hat{w} \rangle}{\| P_W \hat{w}\|_{W}}.
  \end{equation}
  Therefore,
  \begin{align*}
    \| q_h \|_Q
     & \leq c_W \frac{\langle q_h , P_W \hat{w} \rangle}{\| P_W \hat{w}\|_{W}}         
     \leq c_W \sup_{ w_h \in W_h } \frac{\langle q_h , w_h \rangle}{\| w_h\|_{W}}\\
    & \leq c_W \sup_{ v_h \in V_h, w_h \in W_h }
    \frac{\langle q_h , \restr{v_h}{\Omega_f} - w_h \rangle}{(\|w_h\|_{W}^2+\|v_h\|_V^2)^{\frac{1}{2}}},
  \end{align*}
  and we conclude as in Proposition \ref{prop:inf-supE}.
\end{proof}

\begin{proposition}
  \label{prop:inf-supFh}
  Assume that $\mathbb{C}_{\Omega}$ and $\mathbb{C}_f$ are strongly elliptic
  with constants $c_\Omega$ and $c_f$ respectively, such that $c_f > c_\Omega >  0$;
  then there exists a constant $\alpha_{4} > 0$, independent of $h$, such that:
  $$
    \inf_{(u_h,w_h) \in \ker(\mathbb{E}_h)}
    \sup_{(v_h,y_h) \in \ker(\mathbb{E}_h)}
    \frac{(\mathbb{C}_\Omega \nabla v_h, \nabla u_h)_{\Omega}
    + (\delta \mathbb{C}_f\nabla y_h, \nabla w_h)_{\Omega_f}}
    { \|(v_h,y_h)\|_{\mathbb{V}} \|(u_h,w_h)\|_{\mathbb{V}}}
    \geq \alpha_{4},
  $$
  where 
  \begin{align*}
    \ker(\mathbb{E}_h) \coloneqq & \left\lbrace \pmb v_h \coloneqq (v_h,w_h) \in \mathbb{V}_h\colon
    \langle q_h,\restr{v_h}{\Omega_f} - w_h \rangle = 0 \qquad \forall q_h \in Q_h \right\rbrace.
  \end{align*}  
\end{proposition}
The proof follows the one of Proposition \ref{prop:inf-supF}.

\paragraph{Error estimate}

Proposition \ref{prop:inf-supE} and \ref{prop:inf-supF} allow us
to apply the theory from Chapter $5$ of \cite{boffi2013mixed}, obtaining the following
error estimate:

\begin{theorem}
\label{th:err_3d3d}
  Consider $\mathbb{C}_\Omega$ and $\mathbb{C}_f$, elastic stress tensors
  satisfying the hypothesis of Proposition \ref{prop:inf-supF},
  the domains $\Omega$ and $\Omega_f$ with the regularity required in
  Section \ref{sec:sec1},
  and $b \in L^2(\Omega)^d$.
  Then the following error estimate holds for $(u,w,\lambda)$, solution to Problem \ref{pb:4},
  and  $(u_h,w_h,\lambda_h)$, solution of Problem \ref{problem:weak3d3d_h}:
  \begin{equation}
    \label{eq:sol_error3d3d}
    \begin{aligned}
      \| u - u_h \|_V
       & + \| w - w_h \|_W
      + \| \lambda - \lambda_h \|_Q
      \leq                 \\
      C_e
       & \left(
      \inf_{v_h \in V_h}  \| u - v_h \|_V
      + \inf_{y_h \in W_h}  \| w - y_h \|_W
      + \inf_{q_h \in Q_h}  \| \lambda - \lambda_h \|_Q
      \right),
    \end{aligned}
  \end{equation}
  where $C_e > 0$, and depends on $\alpha_3,\alpha_4,c_\Omega,c_f$ and the norm of the operators
  $\mathcal{K}_\Omega$ and $\mathcal{K}_f$.
\end{theorem}

We remark how the constant $C_e$
is affected by the coupling between the two meshes; as intuition suggests, the quality
of the solution does not depend only the on the ability of $V$ and $W$ to individually
describe it, but also on the coupling between them.

\paragraph{Non-matching meshes}

One of the basic assumptions made in the continuous case is that, for every element
$v \in V$, we have that $\restr{v}{\Omega_f} \in W$; similarly every element $w \in W$
can be extended to an element of $V$.

With an independent discretization of the two meshes, the inclusion $W_h
  \subset V_h$ is in general false. In this case, the two requirements for a good
approximation of the problem are that the projection $P_W \colon W
  \rightarrow W_h$ is $H^1$-stable, and that the kernel of $\mathbb{E}_h$ is
``rich enough''.

The $H^1$ stability condition can be easily obtained by inverse inequalities on
quasi-uniform meshes. For more general meshes, the stability of the $L^2$
projection has been investigated, for example, in~\cite{Crouzeix2001,
bramble2002stability,Bank2013}. The non-matching nature of the discretization
does not deteriorate the approximation properties of the underlying spaces,
provided that the two discretizations have comparable local mesh sizes.

In particular, the triangle inequality and Bramble-Hilbert lemma imply readily
that for any $u$ in $V$ we can write
\begin{equation}
  \begin{split}
    \left\|\restr{u}{\Omega_f} -
    P_W\left(\restr{(P_V u)}{\Omega_f}\right)\right\|_{\Omega_f}
    \leq  &
    \left\|\restr{u}{\Omega_f} - \restr{(P_V u)}{\Omega_f} \right\|_{\Omega_f} + \\
    &
    \left\|\restr{(P_V u)}{\Omega_f} - P_W\left(\restr{(P_V u)}{\Omega_f}\right) \right\|_{\Omega_f} \\
    \leq  & C_V h_V  \left|\restr{u}{\Omega_f}\right|_{H^1(\Omega_f)} +
    C_W h_W \left|\restr{(P_V u)}{\Omega_f}\right|_{H^1(\Omega_f)} ,
  \end{split}
\end{equation}
i.e., even with non-matching meshes, the passage through the two non-matching
discretizations still allows one to control in a straight forward way the
$L^2$ norm of the error for any $H^1$ function on $\Omega$. Provided that the 
discretizations on $\Omega$ and $\Omega_f$ have comparable
mesh sizes (respectively $h_V$ and $h_W$ in the equation above), then the passage
through non-matching meshes keeps the error in the same order.

Moreover, if we integrate exactly on the non-matching grids, it is possible to guarantee
that globally constant and linear functions are included in the kernel,
ensuring that  $\ker (\mathbb{E}_h)~\neq~\lbrace (0,0) \rbrace$.

\section{Thin fibers}
\label{sec:sec3}

The computational cost of discretizing numerous three-dimensional
fibers might render Problem~\ref{problem:weak3d3d_h} too computationally
demanding: a possible simplification is to approximate the fibers with
one-dimensional structures.
This construction is a non-trivial because the restriction (or trace) of a
three-dimensional function to a one-dimensional domain is not well defined in
the $H^1$ Sobolev space, which is the natural space where the elasticity
problem is well posed.

Instead of resorting to weighted Sobolev spaces and graded meshes, as
done in \cite{d2008coupling, d2012finite}, we define the coupling between 
the one-dimensional fibers and the three-dimensional elastic matrix through 
an averaging technique that renders the problem well posed.

To simplify the exposition, we consider a single fiber with constant radius
$a$; the same results hold for a finite collection of fibers. Given a
one-dimensional curve $\Gamma$ immersed in $\mathbb{R}^3$, we call \emph{fiber}
its tubular extension of radius $a>0$: $$ \Omega_f \coloneqq \lbrace x \in
\mathbb{R}^3 \colon \operatorname{dist}(x,\Gamma) \leq a \rbrace, $$ such that
$\partial \Omega_f$ is non-intersecting.

Given a point $x$ on $\Gamma$, we call $D_a(x)$ the disk of radius $a$
orthogonal to $\Gamma$, and we assume that the fibers can be written as the
image of a diffeomorphism
\begin{equation}
\label{eq_Phi_transform}
  \Phi \colon \Omega_a \rightarrow \Omega_f,
\end{equation}
where $\Omega_a$ is a straight cylinder of radius $a$ and of axis $\Upsilon$ aligned with the coordinate $x_1$, and $\Phi$ satisfies the following hypotheses:
\begin{enumerate}[i)]
  \item $\Phi(\Omega_a) = \Omega_f$
  \item $\Phi \left( \Upsilon \right) = \Gamma$
  \item  $D_{a}(\Phi(x)) = \Phi\left(D_a(x) \right)$, for every $x \in \Upsilon$.
\end{enumerate}

Given two vectors $u,v \in \mathbb{R}^n$, their tensor product is
a matrix of size $n \times n$, denoted as $u \otimes v$;  
defined for each index $i, j \in \mathbb{N}$, $1 \leq i, j \leq n$ as:
$$
(u \otimes v)_{i,j} \coloneqq u_i v_j.
$$

Let $W \coloneqq H^1(\Omega_f)^d$, and $W_\Gamma \coloneqq H^1(\Gamma)^d$. For $w \in W_\Gamma$, the surface gradient along the curve is defined as:
\begin{equation}
  \label{def:derivative_on_curve}
  \nabla_{\Gamma} w
  \coloneqq
  t \otimes  t \nabla \varphi_w,
\end{equation}
where $t$ is the unitary tangent vector
to the curve $\Gamma$ and $\varphi_w$ is a smooth extension of $w$ in a tubular neighborhood of $\Gamma$.

The following result can be used to define the coupling between the three
dimensional elastic matrix and the one dimensional fibers: 

\begin{theorem}
  \label{proposition:ineq_for_infsup}

  Let $\Omega_f$ be a fiber of radius $a$ with center line $\Gamma$. 
  
  The operator $\restrW \colon W \rightarrow W_\Gamma$, defined as
  \begin{equation}
    \label{eq:avg_slip}
    \restrW u (x)
    \coloneqq \frac{1}{|D_a(0)|} \int_{D_a(x)} u(y) dD_y, \qquad x \in \Gamma
  \end{equation}
  is bounded and continuous.
  
\end{theorem}

\begin{proof}

  We start by considering straight cylinders, and smooth functions. Let $u \in C^\infty( \Omega_a)$;
  the following inequalities hold:
  \begin{align}
      \label{eq:cinfty-l2}
    \| \restrW u \|_{\Upsilon} 
    & \leq \frac{1}{\sqrt{\pi} a} \| u \|_{\Omega_a}          \\
    \label{eq:cinfty-h1}
    \|\nabla_{\Upsilon} \restrW u \|_{\Upsilon}
    & \leq \frac{1}{\sqrt{\pi} a^2} \| \nabla u \|_{\Omega_a},
  \end{align}
  where $\Omega_a$ is a straight cylinder of axis $\Upsilon$.

Consider the coordinates $(x_1, x_2, x_3)$, where $x_1$ is aligned with the
cylinder's axis and $x_2, x_3$ are normal to the axis. The first inequality can
be proven directly, using either Jensen's or H\"older inequality:
\begin{align*}
  \| \restrW u \|_{\Upsilon} ^2 & = \frac{1}{(\pi a^2)^2}
  \int_{\Upsilon}  \left( \int_{D_a(x_1)} u(x_1,x_2,x_3) dx_2 dx_3 \right)^2 dx_1                                  \\
                          & \leq \frac{1}{\pi a^2} \int_{\Upsilon}  \int_{D_a(x_1)} u(x_1,x_2,x_3)^2 dx_2 dx_3 dx_1
  = \frac{1}{\pi a^2} \| u \|_{\Omega_a}^2
\end{align*}

The second inequality can be proven in a similar fashion, by splitting the
integral used to compute $\restrW u(x)$ on $\Gamma$ and on the domain $D
\coloneqq \lbrace x_2,x_3 \text{ s.t. } x_2^2+x_3^2 \leq a^2 \rbrace$, which
does not depend on $x_1$, and observing that $|\nabla_{\Upsilon} \restrW
u(x)|^2 = |\partial_{x_1} \restrW u (x)|^2$.

From~\eqref{eq:cinfty-l2} and \eqref{eq:cinfty-h1}, we 
conclude with a density argument that 
  $$
    u \in H^1(\Omega_a)^d
    \Rightarrow
    \restrW u\in H^1(\Upsilon)^d.
  $$

The generalization to a generic fiber $\Omega_f = \Phi(\Omega_a)$ follows 
applying a change of coordinate transformation through $\Phi$. In particular, 
for any $u \in H^1(\Omega_a)$:
  \begin{align}
    \| u \circ \Phi \|_{H^1(\Omega_a)}        & \leq C_J^{1/2} C_\Phi
    \| u \|_{H^1(\Omega_f)}                               \\
    \| \restrW u \|_{H^1( \Gamma )} & \leq
    (C_J C_\Phi)^{3/2} \| \restrW (u \circ \Phi ) \|_{H^1(\Upsilon)},
  \end{align}
  where we defined the following positive constants:
  $$
    C_J \coloneqq \sup_{z \in \Omega_f} |J\Phi^{-1}(z)|
    \qquad
    C_\Phi \coloneqq \sup_{z \in \Omega_f} |\nabla \Phi(z)|
    .
  $$

  For the first inequality, we begin
  with the $L^2$ part of the norm:
  $$
    \| u \circ \Phi \|_{\Omega_a}^2
    =
    \int_{\Omega_a} (u \circ \Phi)^2 d\Omega_a
    =
    \int_{\Omega_f} u^2 |J\Phi^{-1}| d\Omega_f
    \leq C_J \| u \|_{\Omega_f}^2.
  $$
  Similarly, for $\| \nabla (u \circ \Phi) \|_{\Omega_a}^2$:
  \begin{align*}
    \| \nabla (u \circ \Phi) \|_{\Omega_a}^2
     & =
    \int_{\Omega_a} (\nabla(u \circ \Phi)(z))^2 d\Omega_a         \\
     & =
    \int_{\Omega_a} (\nabla \Phi(z)^T \nabla u(\Phi(z)) )^2 d\Omega_aa
    \leq C_\Phi^2 \int_{\Omega_a} (\nabla u(\Phi(z)))^2 d\Omega_a \\
     & =
    \int_{\Omega_f}
    \left(
    \nabla u(\tilde{z})
    \right)^2
    |J\Phi^{-1}| d\Omega_f
    \leq C_J C_\Phi^2 \| u \|_{\Omega_f}^2.
  \end{align*}

  To prove the second inequality we follow a similar procedure,
  splitting the integral over the centerline of the tubular neighbourhood and
  the restrictions of  $\Phi$ to the perpendicular disks.
  The combined inequalities imply the thesis:
  \begin{equation}
      \| \restrW u \|_{H^1( \Gamma )} \leq C \| u \|_{H^1(\Omega_f)}.
  \end{equation}
\end{proof}

A possible right inverse of the restriction operator is the extension operator $\extV$:
\begin{align*}
\label{eq:extension_operator}
\extV  \colon H^1(\Gamma)^d &\rightarrow H^1(\Omega_f)^d,\\
w &\rightarrow w \circ P_\Gamma,
\end{align*}
where $P_\Gamma$ is the geometric projection from the domain $\Omega_f$
to $\Gamma$, i.e., for every $y \in \Omega_f$, $P_\Gamma y$
is the element of $\Gamma$ such that:
$$
\operatorname{dist}(y,P_\Gamma y) \leq \operatorname{dist}(y,x)
\qquad
\text{for every } x \in \Gamma.
$$
Equation~\eqref{eq_Phi_transform} guarantees that $P_\Gamma$ is well-defined.
The operator $\extV $ is clearly bounded in $L^2$; to prove that its gradient is
bounded it is sufficient to adapt the proof of Theorem~\ref{eq:general_cylinder}.

The following result allows us to define a function $g$ on $\Gamma$,
describing the geometry of $\Omega_f$, and reduce our computations to
a linear integration.

\begin{theorem}
\label{eq:general_cylinder}
For a general tubular neighbourhood $\Omega_f$ of a curve $\Gamma$,
for which the curvature $\kappa$ and the torsion $\tau$ are defined a.e., there exists
a function $g$ defined on $\Gamma$ satisfying:
\begin{equation}
\label{eq:general_cylinder_E}
\frac{1}{2}(\delta \mathbb{C}_f \nabla \extV w,\nabla \extV w)_{\Omega_f}
= \frac{c_\Gamma}{2} (g \delta \mathbb{C}_f \nabla_\Gamma w,\nabla_\Gamma w)_{\Gamma},
\end{equation}
for every $w \in H^1(\Gamma)^d$, with $c_\Gamma \coloneqq \pi a^2$.
\end{theorem}

\begin{proof}

  Let $I$ be an interval, and let $\omega \colon I \rightarrow \Gamma$ be the
  arclength parametrization of $\Gamma$. From the hypotheses, it is possible to
  define an orthonormal basis on $\Gamma$: the tangent $t(s)$, the normal
  $n(s)$, and the binormal $b(s)$, along with the curvature $\kappa(s)$ and the
  torsion $\tau(s)$ (see, e.g., \citep{kreyszig}).
  
  We consider a change of coordinates w.r.t. $t$, $n$, and $b$, which is
  typically used in physics to study wave propagation, optics and particle
  trajectories \cite{lee2018accelerator, tang1970orthogonal, ryder1993nonlinear}, based on the coordinates $(r,\vartheta,s)$, with the orthogonal metric:
  $$
    dx \cdot dx =
    dr^2 + r^2 d\vartheta^2+
    (1-\kappa r \cos(\vartheta - \hat{\theta}))^2 ds^2,
  $$
  and the explicit formula for the gradient:
  \begin{equation}
  \label{eq:gradient_polar}
    \nabla \cdot \coloneqq
    \frac{d\ \cdot }{dr}  e_1
    + \frac{1}{r} \frac{d\ \cdot }{d\vartheta}  e_2
    + \frac{1}{1-\kappa(s) r \cos(\vartheta - \hat{\theta})} \frac{d\ \cdot }{ds}  t.
  \end{equation}
  For a detailed description of these formulas, and the exact definition of
  the angles $\vartheta$ and $\hat{\theta}$ see~\cite[Chapter 3]{ryder1993nonlinear}.
  
Given $w \in  H^1(\Gamma)^d$ the extension $\extV w$ is constant
on each orthogonal disk, i.e., when computing the gradient with
Equation~\eqref{eq:gradient_polar} the components $e_r$ and $e_{\theta}$ are $0$.

Thus:
\begin{align*}
(\delta \mathbb{C}_f \nabla \extV w,\nabla \extV w)_{\Omega_f}
&=
\int_{\Omega_f}
\delta \mathbb{C}_f \nabla \extV w \nabla \extV w 
d\Omega\\
&=
\int_\Gamma
\left(
\int_{D_a(s)} \delta \mathbb{C}_f \nabla_\Gamma w(s) \nabla_\Gamma w(s)
dr d\theta
\right)
ds\\
&=
\int_\Gamma
\left(
\int_{D_a(s)} \frac{1}{1-\kappa(s) r \cos(\theta - \hat{\theta})}
\delta \mathbb{C}_f \frac{d}{ds} w(s)\frac{d}{ds} w(s)
dr d\theta
\right)
ds\\
&=
\int_\Gamma
\delta \mathbb{C}_f \nabla_\Gamma w(s) \nabla_\Gamma w(s)
\underbrace{
\left(
\int_{D_a(s)} \frac{1}{1-\kappa(s) r \cos(\theta - \hat{\theta})}
dr d\theta
\right)
}_{=\colon g(s) \pi a^2 }
ds\\
&=
(g\delta \mathbb{C}_f \nabla_\Gamma w,
\nabla_\Gamma w)_{\Gamma}.
\end{align*}

Where we defined the function:
$$
g(s) \coloneqq \frac{1}{\pi a^2} \int_{D_a(s)}
\frac{1}{1-\kappa(s) r \cos(\theta - \hat{\theta})}
dr d\theta.
$$
\end{proof}

The function $g$ encodes a geometrical stiffness information, which reflects
the fact that rods with finite thickness store more energy (i.e., the function
$g$ grows) when they are bent with high curvature w.r.t. to their thickness 
(i.e., when $\kappa r$ is close to one). In the case of a straight cylinder $\kappa(s) = 0$ a.e., and $g$ reduces to the constant one.

The condition $r \leq a < \max_s 1/\kappa(s)$, which is necessary to guarantee
that $\Omega_f$ is non
self-intersecting~\citep{duistermaat2004multidimensional}, plays a major role
also in defining the geometrical stiffness $g$ on the whole disk. Under these
hypotheses it is also possible to prove that there exists a constant $C_g > 0$
such that $g(s) \geq C_g > 0$.

\subsection{Problem Formulation}

When the fibers are thin compared to the domain size 
(i.e., $a \ll \text{diam}(\Omega))$, it is reasonable to make the following
assumptions:
\begin{enumerate}[i)]
  \item during the deformation, the radius of the fibers does not change;
  \item the deformation field inside the fiber is  
  similar to the deformation on the centerline of the fiber itself.
\end{enumerate}

If we consider then a solution $u$ to Problem~\eqref{eq:classicweak}, we expect that
\begin{equation}
  \label{eq:wconst}
  \extV \restrW u \simeq \restr{u}{\Omega_f},
\end{equation}
jusifying the replacement of the space $W\subseteq H^1(\Omega_f)^d$ with the 
much smaller space $\extV W_\Gamma \subset W \subseteq H^1(\Omega_f)^d$.

If we restrict functions in $W$ in Problem~\eqref{problem:weak3d3d_h} to the
subspace $\extV W_\Gamma$, and exploit Theorem~\ref{eq:general_cylinder} and
the fact that $\restrW \extV w = w$ for any $w\in W_\Gamma$, we obtain the
following energy functional for thin fibers
\begin{align*}
  \psi_T(u,w,\lambda) & = \frac{1}{2}(\mathbb{C}_\Omega \nabla u,\nabla u)_\Omega
  +  c_\Gamma (g \delta \mathbb{C}_f \nabla_\Gamma w,\nabla_\Gamma w)_{\Gamma}
  + \langle q, \restrW u -w \rangle
  - (b,u)_\Omega,
\end{align*}
where the non-slip condition on $\Omega_f$ has been replaced by an \emph{average} non-slip condition:
$$
  \langle q,\restrW u -w \rangle = 0 \qquad \forall\ q \in Q_\Gamma,
$$
where $Q_\Gamma \coloneqq W_\Gamma'$, and here and below the notation $\langle \cdot, \cdot \rangle$ is used to indicate the duality product between $Q_\Gamma'=W_\Gamma$ and $Q_\Gamma := W_\Gamma'$.

We obtain the following formulation for the coupling of thin fibers with a
thick elastic matrix:

\begin{problem}[1D-3D Weak Formulation]
\label{pb:pb6}
Given $b \in L^2(\Omega)^d$,
find $u \in V, w\in W_\Gamma, \lambda \in Q_\Gamma$ such that:
\begin{align*}
  (\mathbb{C}_\Omega\nabla u,\nabla v)_\Omega
  + \langle \lambda,\restrW v \rangle
                                   & = (b,v)_\Omega \qquad & \forall v \in V,          \\
  c_\Gamma (g \delta \mathbb{C}_f \nabla_\Gamma w, \nabla_\Gamma y)_{\Gamma}
  - \langle \lambda,y  \rangle
                                   & = 0 \qquad            & \forall y \in W_{\Gamma}, \\
  \langle q,\restrW u - w  \rangle & = 0 \qquad            & \forall q \in Q_\Gamma.
\end{align*}
\end{problem}

Similarly to what was done in the full order problem, we define the space $\mathbb{V}\coloneqq V \times W_{\Gamma}$,
its norm $\| (\cdot, \cdot) \|_{\mathbb{V}} \coloneqq \| \cdot \|_{V} + \| \cdot \|_{W_\Gamma}$,
and the operators:
\begin{align*}
  \mathbb{F}_T & \coloneqq (\mathbb{C}_\Omega\nabla u,\nabla v)_\Omega +
  c_\Gamma (g \delta \mathbb{C}_f \nabla_\Gamma w,\nabla_\Gamma w)_{\Gamma}, \\
  \mathbb{E}_T & \coloneqq \langle q,\restrW u - w \rangle.
\end{align*}

\subsection{Well-posedness, existence and uniqueness}

\begin{proposition}
  \label{prop:inf-supE1d}
  There exists a constant $\alpha_5 > 0$ such that:
  $$
    \inf_{q \in Q_\Gamma} \sup_{(v,w) \in \mathbb{V}}
    \frac{\langle q,\restrW v- w \rangle}
    { \|(v,w)\|_{\mathbb{V}} \|q\|_{Q_\Gamma}}
    \geq \alpha_5,
  $$
  where:
\begin{equation*}
  \ker{\mathbb{E}_T} \coloneqq \left\lbrace \pmb v \coloneqq (v,w) \in \mathbb{V}\colon
  \langle q,  \restrW v - w \rangle = 0 \  \forall q \in Q_\Gamma \right\rbrace.
\end{equation*}
  Moreover $\alpha_5 = 1$.
\end{proposition}

\begin{proof}
  The non slip condition is given by the duality pairing between
  $Q_\Gamma = W_\Gamma'$ and $W$;
  by definition of the norm in the dual space $Q_\Gamma$:
  \begin{align*}
    \|q\|_{Q_\Gamma} = \sup_{ w \in W_\Gamma }  \frac{ \langle q , w \rangle}{\| w\|_{W}}
    \leq \sup_{v \in V, w \in W_\Gamma }
    \frac{\langle q ,  \restrW v - w \rangle}
    {(\|w\|_{W}^2+\|v\|_V^2)^{\frac{1}{2}}},
  \end{align*}
  where the last inequality can be proven fixing $v=0$.

  The final
  statement is found dividing by $\|q\|_{Q_\Gamma}$ and taking the $\inf_{q \in Q}$.
\end{proof}

\begin{proposition}
  \label{prop:inf-supF1d}
  Assume $\mathbb{C}_{\Omega}$ and $\mathbb{C}_f$ are strongly elliptic,
  with constants $c_\Omega$ and $c_f$ respectively such that $c_\Omega > c_f>0$;
  there exists a constant $\alpha_6 > 0$ such that:
  $$
    \inf_{(u,w) \in \ker{\mathbb{E}_T}}
    \sup_{(v,y) \in \ker{\mathbb{E}_T}}
    \frac{(\mathbb{C}_\Omega \nabla v, \nabla u)_{\Omega}
    +
    c_\Gamma (g\delta \mathbb{C}_f \nabla_\Gamma y,\nabla_\Gamma w)_{\Gamma}}
    { \|(v,y)\|_{\mathbb{V}} \|(u,w)\|_{\mathbb{V}}}
    \geq \alpha_6.
  $$
\end{proposition}

\begin{proof}

  For every $(u,w) \in \ker \mathbb{E}_T$:
  \begin{align*}
    \sup_{(v,y) \in \ker(\mathbb{E}_T)}
     & \frac{
    (\mathbb{C}_\Omega \nabla u, \nabla v)_\Omega
    + c_\Gamma(g\delta \mathbb{C}_f\nabla_\Gamma w, \nabla_\Gamma y)_{\Gamma}
    }
    {\|(v,y)\|_{\mathbb{V}}}                 \\
     & \geq
    \frac{
    c_\Omega  (\nabla u,\nabla u)_{\Omega}
    + c_\Gamma C_g (c_f - c_\Omega ) (\nabla_\Gamma w, \nabla_\Gamma  w)_{\Gamma}
    }
    {\|(u,w)\|_{\mathbb{V}}}                 \\
     & \geq
    \frac{c_p c_\Gamma c_b
      \min(c_\Omega,  C_g (c_f - c_\Omega ))
    }{2}
    \frac{
    (u, u)_{H^1(\Omega)}
    +  (w, w)_{H^1(\Gamma)}
    }
    {\|(u,w)\|_{\mathbb{V}}}                 \\
     & \geq \alpha_6 \|(u,w)\|_{\mathbb{V}},
  \end{align*}
  with $\alpha_6 \coloneqq \frac{c_p c_\Gamma c_b
      \min(c_\Omega,  C_g (c_f - c_\Omega ))
    }{2}$.

  The result is obtained dividing and considering the
  $\inf_{(u,w) \in \ker{\mathbb{E}_T}}$.
\end{proof}

Using the saddle point theory we conclude that Problem~\ref{pb:pb6}
is well-posed, and has a unique solution.

\subsection{Finite element discretization}
\label{subsec:FEMThinFibers}
The discretization of Problem \ref{pb:pb6} for thin fibers
follows the steps of Section \ref{subsec:FEDiscretization},
on the domains $\Omega$ and $\Gamma$ (the fiber's
one-dimensional core).
Consider two independent discretizations for these domains;
the family $\mathcal{T}_h(\Omega)$ of regular meshes
in $\Omega$, and a family $\mathcal{T}_h(\Gamma)$
of regular meshes in $\Gamma$.
We assume no geometrical error is committed when meshing
and consider two independent finite element discretizations
$V_h\subset V,\ W_h\subset W_\Gamma$, and let $Q_h = W_h$.

\begin{problem}[1D-3D Discretized Weak Formulation]
\label{pb:weakTubularDiscr}
Given $b \in L^2(\Omega)^d$,
find $u_h~\in~V_h, w_h\in W_h, \lambda_h \in Q_h$ such that:
\begin{align}
  \label{eq:WeakForm1d3dDiscr}
  (\mathbb{C}_\Omega\nabla u_h,\nabla v_h)_\Omega + \langle \lambda_h,\restrW v_h \rangle
                               & = (b,v_h)_\Omega \qquad & \forall v_h \in V_h, \\
  c_\Gamma (g \delta \mathbb{C}_f \nabla_\Gamma w_h, \nabla_\Gamma y_h)_{\Gamma}
  - \langle \lambda_h, y_h \rangle
                               & = 0 \qquad              & \forall y_h \in W_h, \\
  \label{eq:pb7n3}
  \langle q_h,\restrW u_h-w_h \rangle & = 0 \qquad              & \forall q_h \in Q_h.
\end{align}
\end{problem}

Again, existence, well posedness and convergence are guaranteed if the
classical inf-sup conditions are satisfied:

\begin{proposition}
  \label{prop:inf-supEht}
  Assume the $L^2$ projection $P_{W_\Gamma} \colon W_\Gamma \rightarrow W_h$
  is continuous and $H^1$-stable, as in Proposition~\ref{prop:inf-supEh};
  then there exists a constant $\alpha_{7} > 0$, independent of $h$, such that:
  $$
    \inf_{q_h \in W_h} \sup_{(v_h,w_h) \in \mathbb{V}_h}
    \frac{(q_h, \restrW v_h - w_h )_{\Gamma}}
    { \|(v_h,w_h)\|_{\mathbb{V}} \|q_h\|_{\Gamma}}
    \geq \alpha_{7}.
  $$
\end{proposition}

\begin{proposition}
  \label{prop:inf-supFht}
  Assume $\mathbb{C}_{\Omega}$ and $\mathbb{C}_f$ are strongly elliptic,
  with constants $c_\Omega$ and $c_f$ respectively such that $c_f> c_\Omega > 0$;
  then there exists a constant $\alpha_{8} > 0$, independent of $h$, such that:
  $$
    \inf_{(u_h,w_h) \in \ker \mathbb{E}_{T,h}}
    \sup_{(v_h,y_h) \in \ker \mathbb{E}_{T,h}}
    \frac{(\mathbb{C}_\Omega \nabla v_h, \nabla u_h)_{\Omega}
    + c_\Gamma (g \delta \mathbb{C}_f\nabla_\Gamma y_h, \nabla_\Gamma w_h)_{\Gamma}}
    { \|(v_h,y_h)\|_{\mathbb{V}} \|(u_h,w_h)\|_{\mathbb{V}}}
    \geq \alpha_{8},
  $$
    where:
\begin{equation*}
  \ker{\mathbb{E}_{T,h}} \coloneqq \left\lbrace \pmb v_h \coloneqq (v_h, w_h) \in \mathbb{V}_h \colon 
   \langle q_h,  \restrW v_h - w_h \rangle = 0 \  \forall q_h \in Q_h \right\rbrace.
\end{equation*}
\end{proposition}

The proofs of these propositions are variations
on the proof Propositions~\ref{prop:inf-supEh} and \ref{prop:inf-supFh}. It is
possible to prove an estimate analogous to the one of Theorem\ref{th:err_3d3d}:
\begin{theorem}
  Consider $\mathbb{C}_\Omega$ and $\mathbb{C}_f$, elastic stress tensors
  satisfying the hypothesis of Proposition \ref{prop:inf-supF1d},
  the domains $\Omega$ and $\Omega_f$ with the regularity required in
 this section,
  and $b \in L^2(\Omega)^d$.
  Then the following error estimate holds for $(u,w,\lambda)$, solution to Problem\ref{pb:pb6},
  and  $(u_h,w_h,\lambda_h)$, solution of Problem \ref{pb:weakTubularDiscr}:
  \begin{equation}
    \label{eq:sol_error3d1d}
    \begin{aligned}
      \| u - u_h \|_V
       & + \| w - w_h \|_{W_\Gamma}
      + \| \lambda - \lambda_h \|_{Q_\Gamma}
      \leq                 \\
      C_e
       & \left(
      \inf_{v_h \in V_h}  \| u - v_h \|_V
      + \inf_{y_h \in W_h}  \| w - y_h \|_{W_\Gamma}
      + \inf_{q_h \in Q_h}  \| \lambda - \lambda_h \|_{Q_\Gamma}
      \right),
    \end{aligned}
  \end{equation}
 where $C_e > 0$, and depends on $\alpha_5,\alpha_6,c_\Omega,c_f$ 
 and the norms of the operators
$\mathcal{K}_{\Omega} \colon V \rightarrow V'$ defined as
$\langle \mathcal{K}_{\Omega} u,\cdot\rangle
\coloneqq (\mathbb{C}_{\Omega}\nabla u, \nabla \cdot)_\Omega$, and
$ \mathcal{K}_f \colon W_\Gamma \rightarrow Q_\Gamma$ defined as
$
 \langle \mathcal{K}_f w,\cdot\rangle
 \coloneqq (\delta \mathbb{C}_f\nabla w, \nabla \cdot)_{\Omega_f}$.
\end{theorem}

\section{Numerical validation}
\label{sec:num-exp}

The analytical solution of Problems \ref{pb:4} and \ref{pb:pb6}, even for simple
configurations, is non-trivial: we chose some FRMs structures which are studied
in literature, and used the known approximated solutions as a comparison for our model.

Using the deal.II library \cite{dealii-2020, dealII90,
  BangerthHartmannKanschat2007, MaierBardelloniHeltai-2016-a,
  SartoriGiulianiBardelloni-2018-a}, and the deal.II step-60 tutorial
\cite{step60} we developed a model for thin fibers proposed in Section~\ref{sec:sec3},
and compared it with the Rule of Mixtures and the Halpin-Tsai
configurations in some pull and push tests.

\paragraph{Numerical Setting}

To solve Problem~\ref{pb:weakTubularDiscr} on a collection of fibers we define:
\begin{enumerate}[i)]
\item $V_h = \text{span} \{v_i\}_{i=1}^N \subset H^1(\Omega)$, finite element
space of dimension $N \in \mathbb{N}$ defined on $\Omega$, the elastic matrix;
\item $W_h = \text{span} \{w_a\}_{i=a}^M \subset H^1(\Gamma)$ the finite
element discretization of dimension $M \in \mathbb{N}$ defined on the collection of
$n_f \in \mathbb{N}$ fibers $\Gamma \coloneqq \bigcup_{k=1}^{n_f}\Gamma_k \subset \Omega$;
\item $Q_h$: the space of the Lagrange multiplier, discretized as $W_h$.
\end{enumerate}
We use $i,j$ as indices for the space $V_h$ and $a,b$
as indices for the space $W_h$, and assume all hypotheses on spaces
and meshes of Section~\ref{subsec:FEMThinFibers} are satisfied.
For each fiber $\Gamma_k$ we define its tubular neighborhood $\Omega_{a,k}$, and define
$$
  \Omega_f \coloneqq \bigcup_{k=1}^{n_f} \Omega_{a,k}.
$$
In our implementation, we compute integrals over $ \Omega_f$ by quadrature
formulas, computing the integration over the orthogonal disks $D_a(x)$ to
$\Gamma$ using the mid point rule.

The restriction of the operators of Problem~\ref{pb:weakTubularDiscr} to the
discrete finite element spaces produce the following sparse matrices:
\begin{equation}
  \begin{aligned}
    A & \colon V_h \rightarrow V_h' \qquad
      & A_{ij}
      & \coloneqq (\mathbb{C}_{\Omega}\nabla v_i, \nabla v_j)_\Omega,                                                                                            \\
    K & \colon W_h \rightarrow W_h' \qquad
      & K_{ab}                                                        & \coloneqq c_\Gamma (g \delta \mathbb{C}_{f}\nabla_\Gamma w_a, \nabla_\Gamma w_b)_\Gamma, \\
    B & \colon V_h \rightarrow Q_h' \qquad
      & B_{ia}                                                        & \coloneqq
    ( \restr{v_i}{\Gamma}, w_a)_\Gamma,                                                                                                                                \\
    L & \colon W_h \rightarrow Q_h' \qquad
      & L_{ab}                                                        & \coloneqq ( w_a, w_b)_\Gamma.
  \end{aligned}
\end{equation}
Here $B$ is the coupling matrix from $V_h$ to $W_h$,
$M$ the mass-matrix of $W_h$.
The discretization of Problem \ref{pb:weakTubularDiscr} becomes:
find $(u,w,\lambda_h) \in V_h \times W_h \times Q_h$ such that
\begin{equation}
  \label{matrix:pb1}
  \begin{pmatrix}
    A & 0  & B^T  \\
    0 & K  & -L^T \\
    B & -L & 0    \\
  \end{pmatrix}
  \begin{pmatrix}
    u \\
    w \\
    \lambda
  \end{pmatrix}
  =
  \begin{pmatrix}
    g \\
    0 \\
    0
  \end{pmatrix}
  ,
\end{equation}
where $g_i \coloneqq (b,v_i)_\Omega$.
We reduce the system through the solution of 
the following linear problems:
\begin{align*}
Kw = L^T \lambda =  L \lambda &\Rightarrow \lambda = L^{-1} K w, \\
Bu = Lw &\Rightarrow w = L^{-1}Bu,
\end{align*}
obtaining:
\begin{equation}
  \label{eq:to_solve}
  (A+B^T L^{-1} K L^{-1} B)u = (A+ P_{\Gamma}^T K P_{\Gamma})u = g,
\end{equation}
where $P_{\Gamma} \coloneqq L^{-1} B$.
Boundary conditions are imposed
weakly, using Nitsche method as done in~\cite{RotundoKimJiang-2016-a}.

\subsection{Model description}

\begin{figure}[hb]
\nocaption
  \begin{subfigure}{0.45\textwidth}
    \centering
    \def\x{3.0/5.0}
    \begin{tikzpicture}[scale = 1.2]
      \draw[very thick] (0,0) -- (0,3+\x);
      \draw[very thick] (0,3+\x) -- (3+\x,3+\x);
      \draw[very thick] (3+\x,3+\x) -- (3+\x,0);
      \draw[very thick] (3+\x,0) -- (0,0);
      \foreach \i in {0,...,5}
        {
          \foreach \j in {0,...,5}
            {
              \filldraw[fill=red] (\i * \x + 0.15, \j* \x + 0.15) circle (0.05);
              \filldraw[fill=red] (\i * \x + 0.45, \j* \x + 0.45) circle (0.05);
            }
        }
    \end{tikzpicture}
    \caption{Section with homogeneous fibers.}
    \label{fig:fibre_omogenee}
  \end{subfigure}
  \hfill
  \begin{subfigure}{0.49\textwidth}%
    \centering
    \begin{tikzpicture}[thick,scale=3.5]
      \coordinate (A1) at (0, 0);
      \coordinate (A2) at (0, 1);
      \coordinate (A3) at (1, 1);
      \coordinate (A4) at (1, 0);
      \coordinate (B1) at (0.3, 0.3);
      \coordinate (B2) at (0.3, 1.3);
      \coordinate (B3) at (1.3, 1.3);
      \coordinate (B4) at (1.3, 0.3);

      \draw[very thick] (A1) -- (A2);
      \draw[very thick] (A2) -- (A3);
      \draw[very thick] (A3) -- (A4);
      \draw[very thick] (A4) -- (A1);

      \draw[dashed] (A1) -- (B1);
      \draw[dashed] (B1) -- (B2);
      \draw[very thick] (A2) -- (B2);
      \draw[very thick] (B2) -- (B3);
      \draw[very thick] (A3) -- (B3);
      \draw[very thick] (A4) -- (B4);
      \draw[very thick] (B4) -- (B3);
      \draw[dashed] (B1) -- (B4);

      \node at (0.5, 0.5) {$\mathbf{2}$};
      \node at (0.8, 0.8) {$3$};
      \node at (1.2, 0.7) {$\mathbf{1}$};
      \node at (0.15, 0.7) {$0$};
      \node at (0.7, 1.15) {$\mathbf{5}$};
      \node at (0.7, 0.15) {$4$};

    \end{tikzpicture}
    \caption{$\Omega$'s faces numbering}
    \label{fig:cubo_fibre}
  \end{subfigure}
\end{figure}

For our simulations we use linear finite elements and 
uniformly refined hexaedral meshes. The elastic matrix is the unitary cube
$\Omega \coloneqq [0,1]^3$, where for boundary conditions we refer
to Figure \ref{fig:cubo_fibre}.
The used elastic tensors are
described in Equations \ref{eq:elasticity_tens}-\ref{eq:elasticity_tensf};
for the model description we
use the following parameters:
\begin{itemize}
  \item $r_\Omega, r_\Gamma$: global refinements of the meshes describing $\Omega$ and $\Gamma$ respectively.
  \item $\lambda_\Omega, \mu_\Omega, \lambda_f, \mu_f$: Lam\'{e} parameters for the elastic matrix
  and the fibers respectively.
  \item $\beta$: fiber volume ratio or representative volume element (RVE),
        i.e., $\beta = |\Omega_f| / |\Omega| $.
  \item $a$: the radius of the fibers.
\end{itemize}

\subsection{Homogeneous fibers}

We consider a unidirectional composite,
where fibers are uniform in properties and diameter, continuous,
and parallel throughout the composite $\Omega$ (see Figure
\ref{fig:fibre_omogenee}).
We compare the results obtained with our model with the ones
obtained using the Rule of Mixtures \citep{hashin1965elastic, agarwal2017analysis}
to approximate the stress-strain equation:
\begin{equation}
  \label{eq:homog}
  S[u] = \frac{1}{2}(\mathbb{C}_\Omega Eu,Eu)_\Omega +
  \beta
  \frac{1}{2} (\delta \mathbb{C}_f E u,E u)_{\Omega}.
\end{equation}
This approximation agrees with experimental tests especially for tensile loads,
and when the fiber ratio $\beta$ is small.

Multiple tests were run, keeping $\beta$ constant,
while increasing the fiber density
and reducing the fiber diameter; we expect this process
to render the coupled model solution increasingly
close to the homogenized one.

\paragraph{Comparing solutions}

\begin{figure}
\nocaption
  \centering
  \begin{subfigure}{0.45\textwidth}
    \centering
    \includegraphics[width=0.9\textwidth]{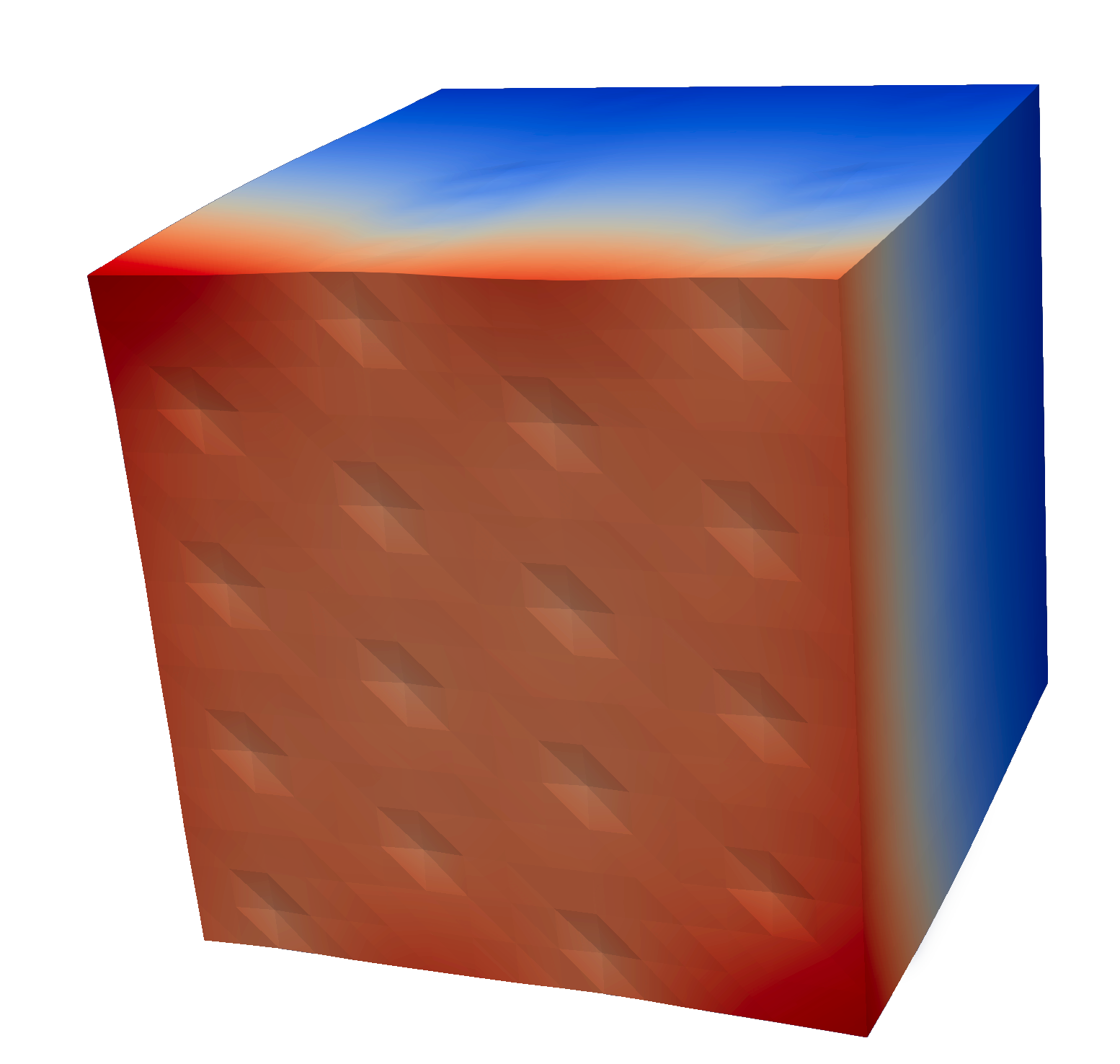} 
    \caption{Pull test example, with $r_\Omega = 4$}
    \label{fig:mesh_refinement4}
  \end{subfigure}\hfill
  \begin{subfigure}{0.45\textwidth}
    \centering
    \includegraphics[width=0.9\textwidth]{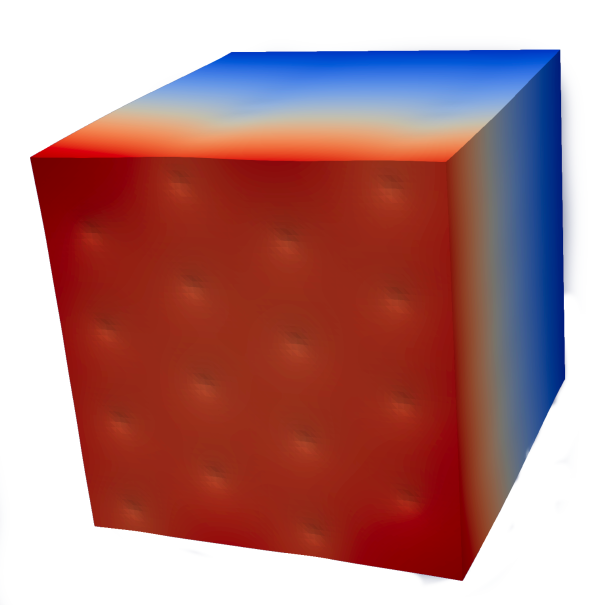} 
    \caption{Pull test example, with $r_\Omega = 6$}
    \label{fig:mesh_refinement6}
  \end{subfigure}
\end{figure}

Figures \ref{fig:mesh_refinement4} and \ref{fig:mesh_refinement6} illustrate
the influence of $\Omega$'s refinement
on the final result, when using few fibers on a pull test.
Stiff fibers oppose being stretched, deforming the
elastic matrix $\Omega$ through the non-slip condition:
near each fiber, the deformation of $\Omega$ should be symmetrical,
resembling a cone. This effect is better described in Figure \ref{fig:mesh_refinement6},
where the higher value of $r_\Omega$ results in greater geometrical flexibility
of the elastic matrix,
allowing a better description of the effect of each fiber.
Lower values of $r_\Omega$ result in a non symmetrical solution,
as in Figure \ref{fig:mesh_refinement4}.
The lower geometrical flexibility results in an ``averaged'' solution which,
in the case of few fibers, is
closer to the homogenized model.

The slopes in the plots have been computed by taking the ratio of the error in
two subsequent simulation, i.e.,
$$
  \frac{
    \ln(\mathbf{e}_2/\mathbf{e}_1)
  }
  {
    \ln(n_{f,2}/n_{f,1})
  },
$$
where $\mathbf{e}_1,\mathbf{e}_2$ be the $L^2$ errors and $n_{f,1},n_{f,2}$ be
the number of fibers for two subsequent simulations.

\paragraph{Pull test along fibers}
\label{par:pull}
Dirichlet homogeneous conditions
is applied to face $0$, Neumann homogeneous conditions is
applied to faces $2,3,4,5$.
In the Push Test, the Neumann condition
$0.05$ is applied to face $1$. For the Pull Test, the value $-0.05$ is
applied to the same face.
Boundary conditions are applied only to $\partial \Omega$;
the fibers interact through the coupling with the elastic matrix.
We report here only data from pull tests, as push tests gave comparable results.

\begin{figure}
\nocaption
  \centering
  \begin{subfigure}{0.45\linewidth}
    \begin{tikzpicture}[scale=0.7]
      \begin{loglogaxis}[xlabel={Number of Fibers},ylabel={L2 Error}]
        \addplot table {data/Full_3/nf_test_full20_1000_0.1_3_1.out};
        \addlegendentry{$(r_\Omega = 3, r_\Gamma = 1)$};

        \addplot table {data/Full_3/nf_test_full20_1000_0.1_3_2.out};
        \addlegendentry{$(r_\Omega = 3, r_\Gamma = 2)$};

        \addplot table {data/Full_3/nf_test_full20_1000_0.1_3_3.out};
        \addlegendentry{$(r_\Omega = 3, r_\Gamma = 3)$};

        \logLogSlopeTriangleReversed{0.2}{.1}{.65}{0.8}{color={rgb:red,115;green,76;yellow,38}}{$0.8$}
      \end{loglogaxis}
    \end{tikzpicture}
    \caption{Pull Test, $\lambda=\lambda_F=0.4, \mu=1,\\ \mu_F=1000, \beta=0.1$.}
    \label{fig:compare_pull_1}
  \end{subfigure}
  \hfill
  \begin{subfigure}{0.45\linewidth}
    \begin{tikzpicture}[scale=0.7]
      \begin{loglogaxis}[xlabel={Number of Fibers},ylabel={L2 Error}]

        \addplot table {data/Full_4/nf_test_full20_1000_0.1_4_3.out};
        \addlegendentry{$(r_\Omega = 4, r_\Gamma = 3)$};

        \addplot table {data/Full_3/nf_test_full20_1000_0.1_3_3.out};
        \addlegendentry{$(r_\Omega = 3, r_\Gamma = 3)$};

        \logLogSlopeTriangleReversed{0.2}{.1}{.63}{0.8}{color=blue}{$0.8$}
      \end{loglogaxis}
    \end{tikzpicture}
    \caption{Pull Test, $\lambda=\lambda_F=0.4, \mu=1,\\ \mu_F=1000, r_\Omega=3, r_\Gamma=3$.}
    \label{fig:compare_pull_3}
  \end{subfigure}
\end{figure}

\begin{figure}
\nocaption
  \centering
  \begin{subfigure}{0.45\linewidth}
  \centering
  \begin{tikzpicture}[scale=0.7]
    \begin{loglogaxis}[xlabel={Number of Fibers},ylabel={L2 Error}]
      \addplot table {data/Full_3/nf_test_full20_100_0.1_3_3.out};
      \addlegendentry{$\mu_F = 100$};

      \addplot table {data/Full_3/nf_test_full20_500_0.1_3_3.out};
      \addlegendentry{$\mu_F = 500$};

      \addplot table {data/Full_3/nf_test_full20_1000_0.1_3_3_lim.out};
      \addlegendentry{$\mu_F = 1000$};

      \logLogSlopeTriangleReversed{0.21}{.1}{.3}{0.84}{color={rgb:red,115;green,76;yellow,38}}{$0.84$}

    \end{loglogaxis}
  \end{tikzpicture}
  \caption{Pull Test with varying $\mu_F$,
    $\lambda = \lambda_F = 0.4, \mu = 1, \beta = 0.1, r_\Omega = 3, r_\Gamma = 3$}
  \label{fig:compare_pull_2}
  \end{subfigure}
  \hfill
  \begin{subfigure}{0.45\linewidth}
  \centering
  \begin{tikzpicture}[scale=0.7]
    \begin{loglogaxis}[xlabel={Number of Fibers},ylabel={L2 Error}]
      \addplot table {data/rand1/4_2_rand.dat};
      \logLogSlopeTriangleReversed{0.3}{.15}{.488}{0.4}{color=blue}{$0.4$}
    \end{loglogaxis}
  \end{tikzpicture}
  \caption{Random Pull Test with one dimensional approximation;
  $r_\Omega=4$, $r_\Gamma=2$.}
  \label{fig:rand_1}
  \end{subfigure}
\end{figure}

The use of the projection matrix $P_\Gamma \colon V_h \rightarrow W_h$,
and the error estimate for the fully three-dimensional case (Inequality
\ref{eq:sol_error3d3d}),
both suggest that the solution quality on the elastic matrix depends on both $V_h$
and $W_h$.
This is apparent in Figure \ref{fig:compare_pull_1}, where
for $r_\Gamma = 1$ the mesh of $\Gamma$ is unable to describe the stretch of the
material, resulting in the error remaining approximately constant after a certain
fiber density is reached. A similar behaviour emerges in the case $r_\Gamma = 2$.
Figure \ref{fig:compare_pull_3}
shows that refining only the elastic
matrix does not improve the solution quality:
as the number of fibers increases, the error converges
to approximately the same value, which is limited by $r_\Gamma$.

Figure \ref{fig:compare_pull_2} shows an error comparison as the value of $\mu_F$
varies: as expected our model is better suited for stiff fibers.

\subsection{Random fibers}

\begin{figure}
  \centering
  \includegraphics[width=0.5\textwidth]{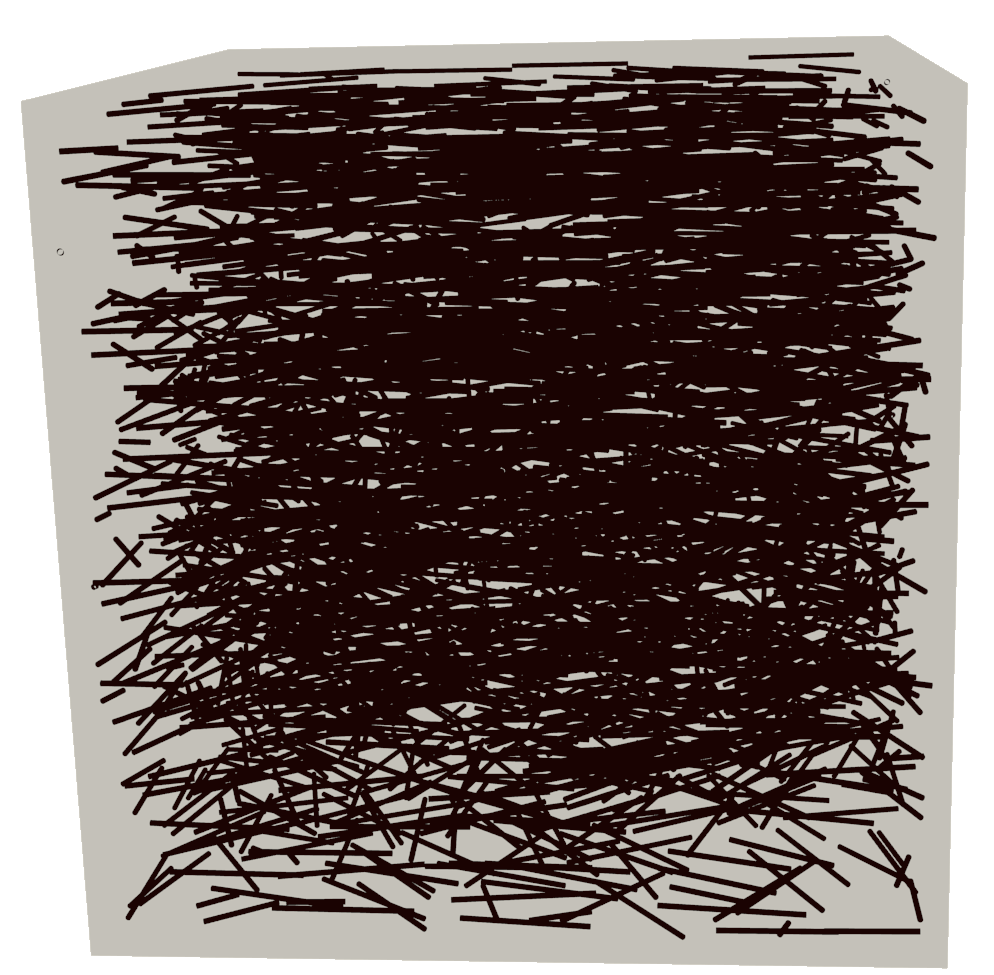} 
  \caption{Random Fibers Disposition}
\end{figure}

We consider a pull test on a random chopped
fiber reinforced composite: we distribute small fibers
at a random points of $\Omega$, with
a random direction parallel to the $\langle x,y\rangle$ plane;
the fibers share the same size and properties.
If a fiber surpasses the edge of $\Omega$, it is cut.

For more details on the algorithm used to distribute
the fibers see the Random Sequential Adsorption algorithm
\cite{pan2008analysis}; our implementation generates only
the plane angle, and does not implement an
intersection-avoidance mechanism.

As a comparison model, we estimate the material parameters using
the empirical Halpin-Tsai \cite{halpin1976halpin}; which
we report here for longitudinal moduli,
as described in \cite{agarwal2017analysis}.
The fibers have length $l$, diameter $d$, the fiber and
the matrix Young moduli are $E_f$ and $E_m$ respectively, $\beta$
is the volume fraction occupied by the fibers.
We define two empirical constants:
$$
\eta_L = \frac{(E_f/E_m)-1}{(E_f/E_m)+(2l/d)}, \qquad
\eta_T = \frac{(E_f/E_m)-1}{(E_f/E_m)+2}.
$$
This allows to compute the longitudinal and transverse moduli for
aligned short fibers:
$$
E_L = \frac{1 + (2l/d) \eta_L \beta }{1 - \eta_L \beta}, \qquad
E_T = \frac{1 + 2 \eta_T \beta }{1 - \eta_T \beta}.
$$
If fibers are randomly oriented in a plane the following equations
can be used to predict the elastic modulus:
$$
E_C = \frac{3}{8} E_L + \frac{5}{8} E_T, \qquad
\mu_C = \frac{1}{8} E_L + \frac{1}{4} E_T.
$$
Since a random fiber composite is considered isotropic in its plane,
the Poisson's ratio can be calculated as:
\begin{equation}
\nu_R = \frac{E_C}{2 \mu_C} - 1.
\end{equation}
According to this approximation, the properties of this composite do not depend directly on
the fiber length or radius, but on the aspect ratio $l/d$.
Our test setting runs on the unitary cube, with a fiber ratio
$\beta = 0.135$ and a fiber aspect ratio of $\frac{l}{2r} \approx 10$,
where $l$ is the fiber's length and $r$ is the fiber's radius; the
values used are described in Table~\ref{tab:beta0.135}.

\begin{table}
\centering
\begin{tabular}{|c|c|c|c|c|c|c|c|c|}
\hline
Fiber length & 0.6 & 0.4 & 0.3 & 0.25 & 0.225 & 0.2 & 0.18\\
\hline
Fiber radius & 0.03 & 0.02 & 0.015 & 0.0125 & 0.01125 & 0.01 & 0.009\\
\hline
Number of Fibers & 79 & 268 & 637 & 1100 & 1509 & 2149 & 2947\\
\hline
\end {tabular}
\captionof{table}{Fiber Parameters} \label{tab:beta0.135}
\end{table}

We could not find an exact estimate of the error
convergence, but we expect the solution to improve
as the number of fibers increases because:
\begin{itemize}
  \item the fiber radius $a$ reduces, improving of the average non-slip
        condition,
  \item more fibers result in a more homogeneous material
        on the planes parallel to the $\langle x,y \rangle$ plane.
\end{itemize}

Following \cite{pan2008analysis}, we consider a short
fiber E-glass/urethane composite:
the fiber and matrix Young's modulus are,
respectively, $E_f~=~70 GPa$ and $E_m~=~3 GPa$,
while the Poisson ratios are $\nu_f~=~0.2$ and
$\nu_m~=~0.38$. These values werex converted
to the Lam\'{e} parameters using the classic
formulas for hyper elastic materials.

The predicted parameters for the composite are:
$E_C~=~2.20 GPa$ and $\nu_C~=~0.38 GPa$;
these are slightly different from \cite{pan2008analysis} because,
in the Halpin-Tsai equations, $l/d$ was used instead of $2l/d$	.
The boundary conditions used for the pull tests are the same
of Paragraph \ref{par:pull}.

We limit the global refinements of $\Gamma$,
in order to obtain cells of approximately the same size on both $\Omega$
and the fibers.
The results are shown in Figure \ref{fig:rand_1}: as the number of
fiber increases the error reduces, but because the random fiber
model is more complex than the homogeneous one,
the final error achieved is higher than the one
reached in the previous test.

\section{Conclusions}
\label{sec:conclusions}

Starting from a linearly elastic description of bi-phasic materials,
we derive a new formulation for fiber reinforced materials where
independent meshes are used to discretize the fibers and the elastic
matrix, and the copling between the two phases is obtained via a
distributed Lagrange multiplier.

We prove existence and uniqueness of a solution for the final saddle
point problem, and we analyze a simplified model where fibers are
discretised as one-dimensional.

The model is validated against the Rule of Mixtures and the Hapin-Tsai
equations, where we test our discretisation with uniform and random
distributions of fibers. The true benefit of our model, however, lies
in the possibility to tackle complex and intricated fiber structures
independently from the background elastic matrix discretization,
opening the way to the efficient simulation of complex multi-phase
materials.

From the numerical analysis point of view, there are some issues that
deserve further development, such as finding better preconditioners for
the final system, and exploring different solutions for the
one dimensional coupling operator.

The formulation of our method makes it particularly suited for
extensions to more complex situations, e.g., three-phasic materials,
or materials where perfect-bond is replaced by other, more realistic
conditions.

\section*{Acknowledgements}

The authors would like to thank Prof. Lucia Gastaldi, for the fruitful
conversations and suggestions that greatly improved this manuscript,
and to thank Dr. Giovanni Noselli, and Dr. Maicol Caponi
for their priceless suggestions and insights in their fields of expertise.

\bibliographystyle{abbrv}
\bibliography{ref}

\end{document}